\documentclass[11pt]{amsart}
\usepackage[T1]{fontenc}
\usepackage{accanthis}
\usepackage[margin=1.35in]{geometry}
\usepackage{amscd,amsmath,amsxtra,amsthm,amssymb,stmaryrd,xr,mathrsfs,mathtools,enumerate,commath, comment}
\usepackage{stmaryrd}
\usepackage{xcolor}
\usepackage{commath}
\usepackage{comment}
\usepackage{tikz-cd}
\usepackage{longtable} 
\usepackage{pdflscape} 
\usepackage{booktabs}
\usepackage{hyperref}
\definecolor{vegasgold}{rgb}{0.77, 0.7, 0.35}
\definecolor{darkgoldenrod}{rgb}{0.72, 0.53, 0.04}
\definecolor{gold(metallic)}{rgb}{0.83, 0.69, 0.22}
\hypersetup{
 colorlinks=true,
 linkcolor=darkgoldenrod,
 filecolor=brown,      
 urlcolor=gold(metallic),
 citecolor=darkgoldenrod,
 pdftitle={Constructing Galois representations with prescribed Iwasawa lambda invariant},
 }

\usepackage[all,cmtip]{xy}

\DeclareFontFamily{U}{wncy}{}
\DeclareFontShape{U}{wncy}{m}{n}{<->wncyr10}{}
\DeclareSymbolFont{mcy}{U}{wncy}{m}{n}
\DeclareMathSymbol{\Sh}{\mathord}{mcy}{"58}
\usepackage[T2A,T1]{fontenc}
\usepackage[OT2,T1]{fontenc}

\newtheorem{theorem}{Theorem}[section]
\newtheorem{lemma}[theorem]{Lemma}

\newtheorem*{theorem*}{Theorem}
\newtheorem*{ass*}{Assumption}
\newtheorem{definition}[theorem]{Definition}

\newtheorem{remark}[theorem]{Remark}

\newtheorem{proposition}[theorem]{Proposition}

\newcommand{\cF}{\mathcal{F}}

\newcommand{\pdiv}{\mid\!\mid}

\newcommand{\cH}{\mathcal{H}}

\newcommand{\Z}{\mathbb{Z}}
\newcommand{\p}{\mathfrak{p}}
\newcommand{\Q}{\mathbb{Q}}
\newcommand{\F}{\mathbb{F}}

\newcommand{\cL}{\mathcal{L}}

\newcommand{\cO}{\mathcal{O}}

\newcommand{\op}[1]{\operatorname{#1}}

\newcommand\mtx[4] { \left( {\begin{array}{cc}
 #1 & #2 \\
 #3 & #4 \\
 \end{array} } \right)}

\numberwithin{equation}{section}

\begin{document}

\title[Constructing Galois representations with prescribed $\lambda$-invariant]{Constructing Galois representations with prescribed Iwasawa $\lambda$-invariant}

\author[A.~Ray]{Anwesh Ray}
\address[Ray]{Chennai Mathematical Institute, H1, SIPCOT IT Park, Kelambakkam, Siruseri, Tamil Nadu 603103, India}
\email{anwesh@cmi.ac.in}


\maketitle

\begin{abstract}
Let $p\geq 5$ be a prime number. We consider the Iwasawa $\lambda$-invariants associated to modular Bloch-Kato Selmer groups, considered over the cyclotomic $\mathbb{Z}_p$-extension of $\mathbb{Q}$. Let $g$ be a $p$-ordinary cuspidal newform of weight $2$ and trivial nebentype. We assume that the $\mu$-invariant of $g$ vanishes, and that the image of the residual representation associated to $g$ is suitably large. We show that for any number $n$ greater than or equal to the $\lambda$-invariant of $g$, there are infinitely many newforms $f$ that are $p$-congruent to $g$, with $\lambda$-invariant equal to $n$. We also prove quantitative results regarding the levels of such modular forms with prescribed $\lambda$-invariant.
\end{abstract}

\section{Introduction}

\par The Iwasawa theory of Selmer groups associated to Galois representations captures significant arithmetic information about motives. Let $p$ be a prime number. Given an elliptic curve $E_{/\Q}$, and a number field extension $F/\Q$, the $p$-primary Selmer group of $E$ over $F$ is of fundamental importance and captures information about the Mordell-Weil group $E(F)$ and the $p$-primary part of the Tate-Shafarevich group $\Sh(E/F)$. The fundamental object of study in the Iwasawa theory of elliptic curves is the $p$-primary Selmer group over the cyclotomic $\Z_p$-extension, denoted $\op{Sel}_{p^\infty}(E/\Q_{\op{cyc}})$. Mazur \cite{mazur1972rational} initiated the Iwasawa theory of elliptic curves $E$ with good ordinary reduction at $p$, and associated structural invariants associated with these Selmer groups.

\subsection{Main results} In this paper we consider $\lambda$-invariants associated to Bloch-Kato Selmer groups attached to modular Galois representations. We prove certain qualitative and quantitative results about the levels of modular forms that arise in natural families, for which the $\lambda$-invariant is prescribed to be a fixed value. We fix a cuspidal Hecke newform $g$ of weight $2$ on $\Gamma_0(N_g)$. Associated with a fixed choice of embedding $\iota_p:\bar{\Q}\hookrightarrow \bar{\Q}_p$, let $\rho_g$ be the associated Galois representation. It is assumed that the image of the residual representation $\bar{\rho}_g:\op{Gal}(\bar{\Q}/\Q)\rightarrow \op{GL}_2(\bar{\F}_p)$ is up to conjugation, equal to $\op{GL}_2(\F_p)$. When we assume that $p$ is odd and $g$ is $p$-ordinary, the Iwasawa $\mu$ and $\lambda$-invariants are well defined, and denoted by $\mu_p(g)$ and $\lambda_p(g)$ respectively (cf. section \ref{s 3.2}).

\begin{theorem}\label{main thm 1}
   Let $p\geq 5$ be a prime and let $g$ be a normalized newform of weight $2$ on $\Gamma_0(N_g)$. We assume the following conditions

\begin{enumerate} 
\item We shall assume that the image of the residual representation $\bar{\rho}_g$ lies in $\op{GL}_2(\F_p)$. Moreover, the Galois representation
$\bar{\rho}_g:\op{Gal}(\bar{\Q}/\Q)\rightarrow \op{GL}_2(\F_p)$ is surjective.
\item The modular form $g$ is $p$-ordinary and $p\nmid N_g$,
\item $g$ has optimal level, i.e., $N_g$ is the prime to $p$ part of the Artin conductor of the residual representation,
\item $\mu_p(g)=0$.
\end{enumerate}
Associated with $g$, let $\Pi_g$ be the set of prime numbers $\ell$ such that the following conditions are satisfied:
\begin{enumerate}
    \item $\ell\nmid N_g p$, 
    \item $\ell\not\equiv \pm 1\mod{p}$ and $\ell^{p-1}\not\equiv 1\mod{p^2}$,
    \item$\bar{\rho}_g(\op{fr}_\ell)=\mtx{\ell}{0}{0}{1}$ for a suitable choice of basis for the residual representation.
\end{enumerate}

Let $\Omega_g$ be the set of prime numbers such that 
\begin{enumerate}
    \item $\ell\nmid N_g p$, 
    \item $\ell\not\equiv \pm 1\mod{p}$,  \item$\bar{\rho}_g(\op{fr}_\ell)=\mtx{-\ell}{0}{0}{-1}$ for a suitable choice of basis for the residual representation. 
    \end{enumerate}
    The primes $\Pi_g$ and $\Omega_g$ have Dirichlet density $\frac{(p-3)}{p(p-1)}$ and $\frac{(p-3)}{(p-1)^2}$ respectively.

Let $n, r\in \Z_{\geq 0}$ be such that $\op{max}\{n, r\}>0$. 
Let $n\geq 1$ be an integer. Then for any sets of primes $\{q_1, \dots, q_n\}\subset \Pi_g$ and $\{\ell_1, \dots, \ell_r\}\subset \Omega_g$, there exists a normalized newform $f$ of weight $2$ of level $N_f=N_g \times \prod_{i=1}^n q_i\times \prod_{j=1}^r \ell_j$ such that 
\begin{enumerate}
    \item $f$ has good ordinary reduction at $p$, 
    \item $\bar{\rho}_g\simeq \bar{\rho}_f$, 
    \item $\mu_p(f)=0$, 
    \item $\lambda_p(f)=\lambda_p(g)+n$.
\end{enumerate}

\end{theorem}

\par In the special case when $\mu_p(g)=0$ and $\lambda_p(g)\leq 1$, one is able to realize an infinite family of modular forms $f$ that are $p$-congruent to $g$, for which the Bloch-Kato Selmer group over $\Q$ has prescribed corank. We refer to Theorem \ref{last thm} for the statement of the result.

\par The following result further illustrates Theorem \ref{main thm 1}.

\begin{theorem}\label{main thm 2}
    There is a positive density set of primes $p$, such that for any $n\geq 0$, there exist infinitely many normalized Hecke cuspdidal newforms $f$ of weight $2$ such that 
    \begin{enumerate}
    \item $f$ has good ordinary reduction at $p$, 
    \item $\bar{\rho}_f$ is surjective, 
    \item $\mu_p(f)=0$, 
    \item $\lambda_p(f)=n$.
\end{enumerate}
\end{theorem}
\subsection{Relationship with previous work} The existence of Galois representations for which the associated Selmer groups have large $\lambda$-invariant has been studied by various authors, cf. \cite{Gre99,ray2023constructing}. Our results are significantly stronger since we are able to explicitly realize any large enough integer as a $\lambda$-invariant. If at a given prime $p\geq 5$, there exists a newform $g$ satisfying the conditions of Theorem \ref{main thm 1}, and such that $\lambda_p(g)=0$, then, every integer $n\geq 0$ is seen to arise from a modular form which is $p$-congruent to $g$. This is indeed the case for a density $1$ set of primes $p$, as shown by the proof of Theorem \ref{main thm 2}. Furthermore, not only is one able to construct infinitely many modular Galois representations giving rise to a prescribed $\lambda$-invariant, one also obtains an explicit and satisfactory quantitative description of their levels. We contrast Theorem \ref{main thm 1} to the results of recent work by Hatley and Kundu \cite{hatley2022lambda}, where it is shown that there are infinitely many modular forms $f$ that are $p$-congruent to a fixed modular form $g$, for which $\lambda_p(f)=\lambda_p(g)$. This $\lambda$-stability result requires an additional assumption on $g$: that the $\lambda$-invariant of $g$ is minimal in the family of all $\lambda$-invariants for modular forms that are $p$-congruent to $g$. For further details, we refer to p.15 of \emph{loc. cit.} This assumption is clearly satisfied when $\lambda_p(g)$ is $0$. This $\lambda$-stability result follows from Theorem \ref{main thm 1} in the special case when $n=0$, without the additional hypothesis. The method used in proving our results does draw some inspiration from \cite{ray2023constructing,hatley2022lambda}.

\subsection{Organization} Including the introduction, the manuscript consists of 5 sections. In section \ref{s 2}, we set up basic notation and review the level raising theorems of Carayol (cf. Theorem \ref{carayol thm}) and Diamond-Taylor (cf. Theorem \ref{DT theorem}). In section \ref{s 3}, we describe the relationship between the Bloch-Kato and Greenberg Selmer groups over the cyclotomic $\Z_p$-extension of $\Q$. We describe the Iwasawa invariants associated to these Selmer groups. In section \ref{s 4}, we recall the results of Greenberg and Vatsal, and prove Theorem \ref{thm 4.9}, which paves the way to the proof of Theorem \ref{main thm 1}. In section \ref{s 5}, we compute the densities of the sets $\Pi_g$ and $\Omega_g$, and prove Theorems \ref{main thm 1} and \ref{main thm 2}.

\subsection*{Acknowledgment} The project was initiated in 2022 when the author was a Simons postdoctoral fellow at the Centre de recherches math\'ematiques in Montreal. From September 2022 to September 2023, the author's research was supported by the CRM Simons postdoctoral fellowship.

\section{Preliminaries}\label{s 2}
\par Fix an algebraic closure $\bar{\Q}/\Q$, and for each prime $\ell$, let $\bar{\Q}_\ell$ be an algebraic closure and fix an inclusion of $\iota_\ell:\bar{\Q}\hookrightarrow \bar{\Q}_\ell$. Set $\bar{\Z}$ to be the integral closure of $\Z$ in $\bar{\Q}$. The choice of embedding $\iota_\ell$ corresponds to a choice of prime $\mathfrak{l}|\ell$ of $\bar{\Z}$. The inclusion $\iota_\ell$ induces an isomorphism of $\bar{\Z}_{\mathfrak{l}}$ with $\bar{\Z}_\ell$. Let $\op{G}_\ell$ denote the absolute Galois group $\op{Gal}(\bar{\Q}_\ell/\Q_\ell)$. The inclusion $\iota_\ell$ induces an inclusion of $\op{G}_\ell$ into $\op{Gal}(\bar{\Q}/\Q)$, whose image is the decomposition group of $\mathfrak{l}$. Let $\op{I}_\ell$ be the inertia group of $\op{G}_\ell$, and choose a Frobenius element $\op{fr}_\ell\in \op{G}_\ell$. Given a Galois representation $\rho:\op{Gal}(\bar{\Q}/\Q)\rightarrow \op{GL}_2(\cdot)$, let $\rho_{|\ell}$ be the restriction of $\rho$ to $\op{G}_\ell$. The representation $\rho$ is \emph{unramified} at $\ell$ if the restriction of $\rho_{|\ell}$ to $\op{I}_\ell$ is trivial. Fix a prime $p\geq 5$ and let $\p|p$ be the prime above $p$, corresponding to the inclusion $\iota_p$. Let $f=\sum_{n=1}^\infty a_n(f)q^n$ be a normalized cuspidal newform of weight $k\geq 2$ and $\rho_f:\op{Gal}(\bar{\Q}/\Q)\rightarrow \op{GL}_2(\bar{\Q}_p)$ be the associated Galois representation. The Hodge-Tate weights for the Galois representation are $\{k-1, 0\}$. We note here that this is simply a matter of convention (instead of the Hodge-Tate weights being $\{0, 1-k\}$). There is a finite extension $K$ over $\Q_p$ such that w.r.t a suitable choice of basis, the image of $\rho_f$ lies in $\op{GL}_2(K)$. In greater detail, let $F$ be a number field containing the Fourier coefficients of $f$, and $K$ is the completion $F_{\mathfrak{p}}$. We set $V_f\simeq K^2$ to denote the underlying Galois module for the representation $\rho_f$. Set $\cO$ to denote the valuation ring of $K$, $\varpi$ be a uniformizer of $\cO$, and let $\kappa:=\cO/(\varpi)$ be the residue field of $\cO$. There exists a Galois stable $\cO$-lattice $T_f\subset V_f$. We shall also denote the integral representation on $T_f$ by $\rho_f:\op{Gal}(\bar{\Q}/\Q) \rightarrow \op{GL}_2(\cO)$. We denote its mod-$\varpi$ reduction by $\bar{\rho}_f:\op{Gal}(\bar{\Q}/\Q) \rightarrow \op{GL}_2(\kappa)$. We shall assume throughout the $\bar{\rho}_f$ is absolutely irreducible. In this setting, it is easy to see that the Galois stable $\cO$-lattice $T_f$ is uniquely determined, and hence, there is no ambiguity in the notation used.
\subsection{The level raising results of Diamond-Taylor}
\par Let $p$ be an odd prime number and $\bar{\rho}:\op{Gal}(\bar{\Q}/\Q)\rightarrow \op{GL}_2(\bar{\F}_p)$ be an irreducible Galois representation. Let $c\in \op{Gal}(\bar{\Q}/\Q)$ denote the complex conjugation; $\bar{\rho}$ is \emph{odd} if $\op{det}\bar{\rho}(c)=-1$. Serre \cite{serre1987representations} conjectured that any odd and irreducible representation is modular. In greater detail, the strong form of the conjecture states that $\bar{\rho}$ arises from a cuspidal newform $g$ of weight $k\geq 2$ on $\Gamma_1(N_{\bar{\rho}})$, where $N_{\bar{\rho}}$ is the prime to $p$ part of the Artin-conductor of $\bar{\rho}$ (cf. \cite[p.11]{ribet1999lectures}). The weight $k=k(\bar{\rho})$ is prescribed according to \cite[section 2]{serre1987representations}. Khare and Wintenberger \cite{khare2009serre, khare2009serre2} proved Serre's conjecture, building upon prior work of Ribet \cite{ribetlevellowering}. Suppose that $f$ is a newform of weight $k$ and level $N_f$ coprime to $p$, such that the associated $p$-adic Galois representation $\rho_f$ lifts $\bar{\rho}$. Then, the optimal level $N_{\bar{\rho}}$ divides $N_f$. A theorem of Carayol \cite{carayol1989representations} proves necessary conditions for an integer $N_f$ to arise in this way from a newform $f$.
\begin{theorem}[Carayol]\label{carayol thm}
    Suppose there exists a modular form $f$ of weight $k$ of level $N_f$ such that $\bar{\rho}=\bar{\rho}_f$ is absolutely irreducible. The level $N_f$ admits a factorization $N_f=N_{\bar{\rho}} \prod_\ell \ell^{\alpha(\ell)}$, 
and for each $\ell$ with $\alpha(\ell)>0$, one of the following holds:
\begin{enumerate}
\item $\ell \nmid N_{\bar{\rho}}$, $\ell \left(\op{trace}\bar{\rho}\left(\sigma_{\ell} \right) \right)^2 = \left( 1+\ell\right)^2 \op{det}\bar{\rho}\left( \sigma_{\ell}\right)$ in $\bar{\F}_p$, and $\alpha(\ell)=1$;
\item $\ell \equiv -1 \mod{p}$ and one of the following holds:
\begin{enumerate}
\item $\ell \nmid N_{\bar{\rho}}$, $\op{trace} \bar{\rho}\left(\sigma_{\ell} \right) \equiv 0$ in $\overline{\F}_p$, and $\alpha(\ell)=2$;
\item $\ell \pdiv N_{\bar{\rho}}$, $\op{det} \bar{\rho}$ is unramified at $\ell$, and $\alpha(\ell)=1$;
\end{enumerate}
\item $\ell \equiv 1 \mod{p}$ and one of the following holds:
\begin{enumerate}
\item $\ell \nmid N_{\bar{\rho}}$ and $\alpha(\ell)=2$;
\item $\ell\pdiv N_{\bar{\rho}}$ and $\alpha(\ell)=1$ or $\ell\nmid N_{\bar{\rho}}$ and $\alpha(\ell)=1$.
\end{enumerate}
\end{enumerate}
\end{theorem}

\begin{definition}
    The set of levels satisfying the conditions outlined in Theorem \ref{carayol thm} is denoted by $\mathcal{S}(\bar{\rho})$.
\end{definition}

\begin{theorem}[Diamond-Taylor \cite{diamond1994non}]\label{DT theorem}
Let $p\geq 5$ be a prime, \[\bar{\rho}:\op{Gal}(\bar{\Q}/\Q)\rightarrow \op{GL}_2(\bar{\F}_p)\] be an irreducible Galois representation. Assume that $\bar{\rho}$ arises from a newform $g$ on $\Gamma_1(N_{\bar{\rho}})$ of weight $k$. Assume that $k$ lies in the range $2\leq k\leq p-2$, and let $M\in \mathcal{S}(\bar{\rho})$. Then there exists a cuspidal newform $f$ of weight $k$ on $\Gamma_1(M)$, such that $\bar{\rho}_f=\bar{\rho}$. 
\end{theorem}
The result of Carayol and the above level-raising result of Diamond and Taylor shows that $\mathcal{S}(\bar{\rho})$ is precisely the set of levels for the cuspidal newforms $f$ such that $\bar{\rho}_f\simeq \bar{\rho}$.

\section{Selmer groups associated to modular forms}\label{s 3}

\par Let $f=\sum_{n=1}^\infty a_n(f) q^n$ be a normalized new cuspform on $\Gamma_1(N_f)$ and $p$ be an odd prime which is coprime to $N_f$. Let $\rho_f:\op{Gal}(\bar{\Q}/\Q)\rightarrow \op{GL}_2(\cO)$ be the associated Galois representation and assume that the residual representation is absolutely irreducible. Here, $K/\Q_p$ is a finite extension and $\cO$ is the valuation ring of $K$. Let $V_f$ and $T_f$ be as defined in previous sections. We shall assume throughout that $p$ is an ordinary prime, i.e., $p\nmid a_p(f)$. Set $\mu_{p^n}$ to denote the $p^n$-th roots of unity in $\bar{\Q}$, and let $\Q(\mu_{p^n})$ be the cyclotomic extension of $\Q$ generated by $\mu_{p^n}$. Note that the Galois group $\op{Gal}(\Q(\mu_{p^n})/\Q)$ is isomorphic to $\left(\Z/p^n\Z\right)^\times$. Let $\Q(\mu_{p^\infty})$ be the union of cyclotomic fields $\Q(\mu_{p^n})$. The Galois group $\op{Gal}(\Q(\mu_{p^\infty})/\Q)$ is isomorphic to $\Z_p^\times$ and decomposes into a product $\Delta\times \Z_p$, where $\Delta$ is isomorphic to $\Z/(p-1)\Z$. The cyclotomic $\Z_p$-extension $\Q_{\op{cyc}}$ is the field $\Q(\mu_{p^\infty})^\Delta$ and the Galois group $\op{Gal}(\Q_{\op{cyc}}/\Q)$ is isomorphic to $\Z_p$. For $n\geq 0$, the \emph{$n$-th layer} $\Q_n/\Q$ is the extension contained in $\Q_{\op{cyc}}$ such that $[\Q_n:\Q]=p^n$.

\subsection{Definition of Selmer groups}
\par Let $S$ be a finite set of rational primes containing $\{p, \infty\}$, and $\Q_S\subset \bar{\Q}$ be the maximal extension in which all primes $\ell\notin S$ are unramified. Setting $A_f:=V_f/T_f$, we note that $A_f\simeq \left(K/\cO\right)^2$. Let $d^\pm$ be the dimension of the $(\pm 1)$-eigenspaces for complex conjugation. We note that $d^+=d^-=1$. Since we assume that $f$ is $p$-ordinary, there exists a $1$-dimensional, $K$-subspace $W_f\subseteq V_f$ of dimension $1$, which is stable under the action of $\op{G}_p$. Let $C$ be the image of $W_f$ in $A_f$, with respect to the mod-$T_f$ map $V_f\rightarrow A_f$, and set $D=A_f/C$. 

\par We assume that the set $S$ contains all primes dividing $N_f$. Let $L$ be an extension of $\Q$ which is contained in $\Q_{\op{cyc}}$. Thus $L$ is the field $\Q_n$ for some $n\geq 0$, or is the field $\Q_{\op{cyc}}$. Given a rational prime $\ell$, set $\ell(L)$ to be the set of primes of $L$ that lie above $\ell$. We note that the set $\ell(\Q_{\op{cyc}})$ is finite. Given a prime $w$ of $\Q_{\op{cyc}}$, let $\Q_{\op{cyc}, w}$ be the union of all completions of $p$-adic fields contained in $\Q_{\op{cyc}}$ at $w$. Given a prime $\ell\in S$, the Greenberg local Selmer condition $\cH_\ell(L, A_f)$ is defined as follows. For $\ell\neq p$, the local condition is defined by \[\cH_\ell(L, A_f):=\prod_{w\in \ell(L)} H^1(L_w,A_f). \] The local condition at $\ell=p$ is defined differently. Let $\eta_p$ be the unique prime of $L$ that lies above $p$; let $\op{I}_{\eta_p}$ denote the inertia subgroup of $\op{Gal}(\bar{L}_{\eta_p}/L_{\eta_p})$. Let 
\[H_{\eta_p}:=\op{ker}\left(H^1(L_{\eta_p}, A_f)\longrightarrow H^1(\op{I}_{\eta_p}, D)\right),\] and set 
\[\cH_p(L, A_f):=H^1(L_{\eta_p}, A_f)/H_{\eta_p}.\]
Following \cite{greenberg1989iwasawa}, the Greenberg Selmer group $\op{Sel}^{\op{Gr}}(L, A_f)$ is then defined as follows
\[\op{Sel}^{\op{Gr}}(L, A_f):=\op{ker}\left(H^1(\Q_S/L,A_f)\longrightarrow \bigoplus_{\ell\in S} \cH_\ell(L, A_f)\right).\]
For $\ell\in S$, $\cH_\ell(\Q_{\op{cyc}}, A_f)$ is a cofinitely generated $\Z_p$-module (cf. \cite[p.37]{greenberg2000iwasawa}). We let $\sigma_\ell(f)$ denote the $\cO$-corank of $\cH_\ell(\Q_{\op{cyc}}, A_f)$. For $\ell\neq p$, we describe an algorithm to compute $\sigma_\ell(f)$. Let $V_f'$ be the quotient $\left(V_{f}\right)_{\op{I}_\ell}$ and $\widetilde{P}_\ell(f;X)$ be the mod-$\varpi$ reduction of $P_\ell(f;X):=\op{det}\left(\op{Id}-\op{fr}_\ell X\mid V_f'\right)$. Let $s_\ell$ be the number of primes of $\Q_{\op{cyc}}$ that lie above $\ell$. Note that $s_\ell$ is the maximal power of $p$ such that $\ell^{p-1}\equiv 1\mod{ps_\ell}$. Let $d_\ell(f)$ be the multiplicity of $\ell^{-1}$ as a root of $\widetilde{P}_\ell(f;X)$. By \cite[Proposition 2.4]{greenberg2000iwasawa}, it follows that $\sigma_\ell(f)$ is given by
\begin{equation}\label{sigma formula}\sigma_\ell(f)=s_\ell d_\ell(f).\end{equation}

\par Given a number field $L/\Q$ contained in $\Q_{\op{cyc}}$, we let $\op{Sel}^{\op{BK}}(L, A_f)$ denote the Bloch-Kato Selmer group associated to $A_f$, cf. \cite{bloch1990functions} or \cite[p.73, l.-1]{ochiai2000control} for the precise definition. The Bloch-Kato Selmer group over $\Q_{\op{cyc}}$ is defined as the direct limit
\[\op{Sel}^{BK}(\Q_{\op{cyc}},A_f):=\varinjlim_n \op{Sel}^{BK}(\Q_{n},A_f).\]

\begin{proposition}\label{BK Greenberg comparison}
With respect to notation above, there is a natural map \[\op{Sel}^{\op{BK}}(\Q_{\op{cyc}}, A_f)\rightarrow \op{Sel}^{\op{Gr}}(\Q_{\op{cyc}}, A_f)\]with finite kernel and cokernel. 
\end{proposition}
\begin{proof}
    The result above is \cite[Corollary 4.3]{ochiai2000control}.
\end{proof}

From the point of view of Iwasawa theory, we may work with either Selmer group $\op{Sel}^{\op{BK}}(\Q_{\op{cyc}}, A_f)$ or $\op{Sel}^{\op{Gr}}(\Q_{\op{cyc}}, A_f)$. It conveniences us to work with the Greenberg Selmer group, and for ease of notation, we simply set \[\op{Sel}(\Q_{\op{cyc}}, A_f):=\op{Sel}^{\op{Gr}}(\Q_{\op{cyc}}, A_f).\]

\subsection{Iwasawa invariants}\label{s 3.2}

\par We set $\Gamma:=\op{Gal}(\Q_{\op{cyc}}/\Q)$ and $\Lambda$ denote the Iwasawa algebra $\cO\llbracket \Gamma\rrbracket:=\varprojlim_n \cO\left[\Gamma/\Gamma^{p^n}\right]$. Given a module $M$ over $\Lambda$, we shall set $M^\vee:=\op{Hom}_{\op{cnts}}\left(M, \Q_p/\Z_p\right)$. A module $M$ over $\Lambda$ is said to be cofinitely generated (resp. cotorsion) if $M^\vee$ is finitely generated (resp. torsion) as a module over $\Lambda$. The Selmer group $\op{Sel}(\Q_{\op{cyc}}, A_f)$ is a cofinitely generated module over $\Lambda$. There are many cases of interest for which it is known to be cotorsion over $\Lambda$, which we shall further discuss in the next subsection. Given $\Lambda$-modules $M$ and $M'$, a \emph{pseudo-isomorphism} is a map of $\Lambda$-modules $M\rightarrow M'$ whose kernel and cokernel both have finite cardinality. Let $M$ be a cofinitely generated and cotorsion $\Lambda$-module, then, as is well known, there is a pseudo-isomorphism 
\[M^\vee\longrightarrow \left(\bigoplus_{i=1}^s \Lambda/(\varpi^{m_i})\right)\oplus \left(\bigoplus_{j=1}^t \Lambda/(f_j(T)^{n_j})\right).\] In the above sum, $s, t, m_i, n_i\in \Z_{\geq 0}$, and $f_j(T)$ is an irreducible distinguished polynomial in $\Lambda$, i.e., an irreducible monic polynomial whose non-leading coefficients are divisible by $\varpi$. The $\mu$-invariant is given by 
\[\mu_p(M):=\begin{cases} \sum_i m_i & \text{ if }s>0;\\
0 & \text{ if }s=0.
\end{cases}\]
On the other hand, the $\lambda$-invariant is given by 
\[\lambda_p(M):=\begin{cases} \sum_j n_j\op{deg}(f_j) & \text{ if }t>0;\\
0 & \text{ if }t=0.
\end{cases}\]
 It follows from results of Kato \cite{kato2004p} that $\op{Sel}(\Q_{\op{cyc}}/A_f)$ is a cotorsion module over $\Lambda$. Then, the $\mu$ (resp. $\lambda$) invariant of $\op{Sel}(\Q_{\op{cyc}}, A_f)$ is denoted by $\mu_p(f)$ (resp. $\lambda_p(g)$). We note that $\mu_p(f)=0$ if and only if $\op{Sel}(\Q_{\op{cyc}}, A_f)$ is cofinitely generated as an $\cO$-module, and in this case, \[\lambda_p(f)=\op{corank}_{\cO}\left(\op{Sel}(\Q_{\op{cyc}}, A_f)\right).\] It follows from Proposition \ref{BK Greenberg comparison} that the $\lambda$-invariant of $\op{Sel}(\Q_{\op{cyc}}, A_f)$ is equal to the $\lambda$-invariant of $\op{Sel}^{\op{BK}}(\Q_{\op{cyc}}, A_f)$. We shall set $\op{rank}^{\op{BK}}(f)$ to denote the $\cO$-corank of the Bloch-Kato Selmer group $\op{Sel}^{\op{BK}}(\Q, A_f)$. 
\begin{proposition}\label{prop 3.2}
    Assume the $\mu_p(f)=0$. The following assertions hold:
    \begin{enumerate}
\item\label{prop 3.2 p1} $\op{rank}^{\op{BK}}(f)\leq \lambda_p(f)$,
\item\label{prop 3.2 p2} $\op{rank}^{\op{BK}}(f)\equiv  \lambda_p(f)\mod{2}$,
\item \label{prop 3.2 p3}
suppose that $\lambda_p(f)\leq 1$, then, $\op{rank}^{\op{BK}}(f)= \lambda_p(f)$.
    \end{enumerate}
\end{proposition}

\begin{proof}
Part \eqref{prop 3.2 p1} follows as a direct consequence of \cite[Theorem A]{ochiai2000control}. Part \eqref{prop 3.2 p2} follows from a standard argument due to Greenberg \cite[Proposition 3.10]{Gre99}, see \cite[p.1288, proof of Theorem 5.7, ll. 9-15]{hatley2019arithmetic}. Part \eqref{prop 3.2 p3} is a direct consequence of \eqref{prop 3.2 p1} and \eqref{prop 3.2 p2}.
\end{proof}

\begin{remark}\label{only remark}We note that the conventions used here are the same as those in \cite[section 2]{greenberg2000iwasawa}. Note that the ring $\cO$ is chosen to be the valuation ring of $F_{\p}$. It is easy to see that the definition of the $\lambda$-invariant $\lambda_p(f)$ is independent of the choice of field $F$, and thus the valuation ring $\cO$. Also, we note that the $\mu$-invariant vanishes for the Selmer group over $\cO$, if and only if it vanishes after base-change by any valuation ring $\cO'/\cO$. 
\end{remark}
\section{Congruence relations between Selmer groups}\label{s 4}
\par We recall some results of Greenberg and Vatsal \cite{greenberg2000iwasawa}, and apply these results to study the structure of Selmer groups of ordinary Galois representations. Let $g$ be a Hecke newform of weight $2$ on $\Gamma_0(N_g)$. Here, $N_g$ is the level of $g$. Assume that the following conditions are satisfied
\begin{enumerate} 
\item $p\nmid N_g$,
\item $g$ is ordinary at $p$,
\item $\bar{\rho}:=\bar{\rho}_g$ is absolutely irreducible,
\item $g$ has optimal level, i.e., $N_g$ is the prime to $p$ part of the Artin conductor of $\bar{\rho}$. 
\end{enumerate}
Let $f$ be a Hecke new cuspform of weight $2$ on such that $\bar{\rho}_f\simeq \bar{\rho}_g$. The understanding here is that the field $F$ is chosen so that all the Fourier coefficients of both $f$ and $g$ are contained in $F$. Remark \ref{only remark} establishes that one may choose any large field $F$ when studying the $\lambda$-invariant or the vanishing of the $\mu$-invariant. Let $N_f$ be the level of $f$. We assume that $p\nmid N_f$. Note that by Theorem \ref{carayol thm}, $N_f\in \mathcal{S}(\bar{\rho})$. We note that $N_g$ divides $N_f$. We note that $f$ has ordinary reduction at $p$ (cf. \cite[Lemma 3.3]{ray2023constructing}). Let $\epsilon_f$ be the nebentype of $f$, and $\bar{\epsilon}_f:\op{Gal}(\bar{\Q}/\Q)\rightarrow \kappa^\times$ its mod-$\varpi$ reduction. We find that $\op{det}\bar{\rho}_g=\bar{\chi}$, and $\op{det}\bar{\rho}_g=\bar{\chi}\bar{\epsilon}_f$ where $\bar{\chi}$ is the mod-$p$ cyclotomic character. Since $\bar{\rho}_f\simeq \bar{\rho}_g$, we find that $\bar{\epsilon}_f=1$. The principal units in $\cO^\times$ are the units that reduce to $1$ modulo $(\varpi)$. It is clear that if a principal unit is a root of unity, then it is equal to $1$. Since $\epsilon_f$ is a finite order character, $\bar{\epsilon}_f=1$ implies that $\epsilon_f=1$. Therefore, $f$ is a weight $2$ cuspform on $\Gamma_0(N_f)$.

\begin{theorem}[Greenberg-Vatsal \cite{greenberg2000iwasawa}]\label{GV thm}
   With respect to above notation, assume that $\mu_p(g)=0$. Then, $\mu_p(f)=0$, and 
    \[\lambda_p(f)-\lambda_p(g)=\sum_{\ell\mid N_f} \left(\sigma_\ell(g)-\sigma_\ell(f)\right),\] where the sum ranges over the primes $\ell$ dividing $N_f$. 
\end{theorem}

We study the numbers $\left(\sigma_\ell(g)-\sigma_\ell(f)\right)$ for primes $\ell|N_f$. Let $A$ be a local $\cO$-algebra, with maximal ideal $\mathfrak{m}_A$, and residue field $A/\mathfrak{m}_A$ isomorphic to $\kappa$. Fix an $\cO$-algebra isomorphism $\varphi_0: A/\mathfrak{m}_A\rightarrow \kappa$. Let $\varphi:A\rightarrow \kappa$ be the map obtained upon composing the reduction modulo $\mathfrak{m}_A$ map, with $\varphi_0$. We set \[\widehat{\op{GL}}_2(A):=\left\{\mtx{a}{b}{c}{d}\in\op{GL}_2(A)\mid  \mtx{\varphi(a)}{\varphi(b)}{\varphi(c)}{\varphi(d)}=\mtx{1}{0}{0}{1}\right\}.\] 
For $n\geq 1$, the map $\varphi:\cO/\varpi^n\rightarrow \kappa$ is the reduction modulo $\varpi$ map.

\begin{definition}Associated with $g$, let $\Pi_g$ be the set of prime numbers $\ell$ such that the following conditions are satisfied:
\begin{enumerate}
    \item $\ell\nmid N_g p$, 
    \item $\ell\not\equiv \pm 1\mod{p}$ and $\ell^{p-1}\not\equiv 1\mod{p^2}$,
    \item$\bar{\rho}_g(\op{fr}_\ell)=\mtx{\ell}{0}{0}{1}$ for a suitable choice of basis for the residual representation. 
\end{enumerate}
Let $\Omega_g$ be the set of prime numbers such that 
\begin{enumerate}
    \item $\ell\nmid N_g p$, 
    \item $\ell\not\equiv \pm 1\mod{p}$,  \item$\bar{\rho}_g(\op{fr}_\ell)=\mtx{-\ell}{0}{0}{-1}$ for a suitable choice of basis for the residual representation. 
\end{enumerate}
\end{definition}

In section \ref{s 5}, we shall introduce further assumptions on the image of $\bar{\rho}_g$, so that the sets $\Pi_g$ and $\Omega_g$ have positive density. For $\ell\in \Pi_g$ (resp. $\Omega_g$), we note that $\rho_g$ is unramified at $\ell$. Since $\bar{\rho}_f\simeq \bar{\rho}_g$, we find that $\bar{\rho}_f$ is unramified at $\ell$ as well. Thus, $\rho_f(\op{I}_\ell)$ is contained in $\widehat{\op{GL}}_2(\cO)$, which is a pro-$p$ group. Hence, $\rho_{f|\ell}$ factors through $\op{Gal}(\Q_\ell^{p-tr}/\Q_\ell)$, where $\Q_\ell^{p-tr}\subset \bar{\Q}_\ell$ is the maximal tamely ramified extension of $\Q_\ell$ with pro-$p$ inertia. The Galois group $\op{Gal}(\Q_\ell^{p-tr}/\Q_\ell)$ is generated by a frobenius $\op{fr}_\ell$, and the pro-$p$ tame inertia generator $\tau_\ell$, subject to the relation $\op{fr}_\ell \tau_\ell \op{fr}_\ell^{-1}=\tau_\ell^\ell$ (cf. for instance \cite[p.123, ll.15-16]{ramakrishna2002deforming}).

\par Let $n\in \Z_{\geq 1}$, and consider the reduction mod-$\varpi^n$ map
\[\pi_n:\op{GL}_2(\cO/\varpi^{n+1})\rightarrow \op{GL}_2(\cO/\varpi^{n}).\]
Given a ring $A$, denote by $\op{M}_2(A)$ the group of $2\times 2$ matrices with entries in $A$. A matrix $M$ in the kernel of $\pi_n$ can be written in the following form
\[M=\op{Id}+\varpi^n\mtx{a}{b}{c}{d}.\]Here, $\mtx{a}{b}{c}{d}$ belongs to $\op{M}_2(\cO/\varpi^{n+1})$. However, since the coefficients of the matrix $M$ belong to $\cO/\varpi^{n+1}$, it follows that $a,b, c,d$ are determined modulo $\varpi$. Thus the matrix $\mtx{a}{b}{c}{d}$ can be assumed to be in $\op{M}_2(\cO/\varpi)$.
\par To illustrate this, consider an example in the case when $n=1$, and $\mtx{a}{b}{c}{d}=\mtx{1}{0}{0}{0}$. We find that 
\[M=\op{Id}+\varpi\mtx{1}{0}{0}{0}=\mtx{1+\varpi}{0}{0}{1}.\] On the other hand, 
\[\op{Id}+\varpi\mtx{1+\varpi}{0}{0}{0}=\mtx{1+\varpi+\varpi^2}{0}{0}{1}=\mtx{1+\varpi}{0}{0}{1}=M,\]since $\varpi^2=0$ when the coefficients are in $\cO/\varpi^2$.  

\par For $i=1, 2$ and $M_i\in \op{ker}\pi_n$, and let $B_i\in M_2(\cO/\varpi)$ be such that $M_i=\op{Id}+\varpi^n B_i$. Then, since $2n\geq n+1$, we have $\varpi^{2n}=0$. Therefore, we find that
\[M_1 M_2=\left(\op{Id}+\varpi^n B_1\right)\left(\op{Id}+\varpi^n B_2\right)=\op{Id}+\varpi^n(B_1+B_2).\]
The inverse of $M=\op{Id}+\varpi^n B$ is given by $M^{-1}=\op{Id}-\varpi^n B$. It is easy to see that the map
\[\op{M}_2(\cO/\varpi)\rightarrow \op{ker}\pi_n,\]mapping $B$ to $\op{Id}+\varpi^n B$ is an isomorphism of abelian groups. Let $B\in \op{ker}\pi_n$ and $M\in \op{GL}_2(\cO/\varpi^{n+1})$. We write $B=\op{Id}+\varpi^n B'$, where $B'=\mtx{a}{b}{c}{d}\in \op{M}_2(\cO/\varpi)$. We write \[[B', M]:=B'\bar{M}-\bar{M}B'\in \op{M}_2(\cO/\varpi),\] where $\bar{M}$ is the reduction of $M$ modulo $\varpi$.  

\begin{lemma}\label{BMB^{-1}lemma}
    With respect to notation above, 
    \[B M B^{-1}=M+\varpi^n [B',M].\]
\end{lemma}
\begin{proof}
    We find that since $\varpi^{n+1}=0$, in particular, $\varpi^{2n}=0$. Therefore, we find that
    \[\begin{split}B M B^{-1} &=(\op{Id}+\varpi^nB')M (\op{Id}-\varpi^nB')\\
    &=M+\varpi^n[B', M].\end{split}\] Note that $\varpi^n [B', \bar{M}]$ is determined by the reduction of $M$ modulo $\varpi$.
\end{proof}

\begin{lemma}\label{lemma 4.4}
   Let $\ell$ be a prime such that 
   \begin{enumerate}
       \item $\ell\nmid N_g p$,
       \item $\ell\nmid \pm 1\mod{p}$, 
       \item $\bar{\rho}_g(\op{fr}_\ell)=\mtx{\eta\ell}{}{}{\eta}$, where $\eta=\pm 1$,
       \item $\rho_f$ is ramified at $\ell$.
   \end{enumerate} Then, up to conjugation by a matrix in $\op{GL}_2(\cO)$,
    \[\rho_f(\op{fr}_\ell)=\mtx{\ell x}{}{}{x}\text{ and } \rho_f(\tau_\ell)=\mtx{1}{y}{}{1},\] where $x\equiv \eta \mod{\varpi}$ and $y\equiv 0\mod{\varpi}$.
\end{lemma}

\begin{proof}
We choose an integral basis for $T_f$ so that $\bar{\rho}_f(\op{fr}_\ell)=\mtx{\eta\ell}{0}{0}{\eta}$. Since $\ell\nmid N_g p$, $\bar{\rho}_g$ is unramified at $\ell$. Since $f$ and $g$ satisfy the congruence $\bar{\rho}_f\simeq \bar{\rho}_g$, it follows that $\bar{\rho}_f$ is unramified at $\ell$. Since $\tau_\ell$ is in the inertia group at $\ell$, we find that $\bar{\rho}_f(\tau_\ell)=\op{Id}$. Thus, for $n=1$, the upper triangular entry is $0$.
\par We show that there is a matrix $A\in \op{GL}_2(\cO)$ so that $\rho:=A \rho_f A^{-1}$ satisfies the above conditions. In fact, the matrix $A$ can be chosen to be in $\widehat{\op{GL}}_2(\cO)$. For ease of notation, we set $\rho':=\rho_f$ and $\rho'_n$ denote the $\mod{\varpi^n}$ reduction of $\rho'$. We shall inductively specify matrices $A_n \in \widehat{\op{GL}}_2(\cO/\varpi^n)$ such that $A_m\equiv A_n \mod{\varpi^n}$ for all integers $m>n$. We shall then let $A$ denote the matrix $(A_n)$ in the inverse limit 
    \[\widehat{\op{GL}}_2(\cO)=\varprojlim_n \widehat{\op{GL}}_2(\cO/\varpi^n).\] Let us set $\rho_n:=A_n \rho_n' A_n^{-1}$. The matrices $A_n$ shall have the property that there exist $x_n, y_n, z_n\in \cO/\varpi^n$ such that $\rho_n(\op{fr}_\ell)=\mtx{\ell x_n}{}{}{z_n}$ and $\rho_n(\tau_\ell)=\mtx{1}{y_n}{}{1}$, where $x_n, z_n\equiv \eta\mod{\varpi}$ and $y_n\equiv 0\mod{\varpi}$. Since $A_m \equiv A_n \mod{\varpi^n}$ for all $m>n$, it shall then follow that $x_m\equiv x_n\mod{\varpi^n}$, $y_m\equiv y_n\mod{\varpi^n}$ and $z_m\equiv z_n\mod{\varpi^n}$. Setting $x,y,z$ to denote the inverse limits $(x_n)$, $(y_n)$ and $(z_n)$ respectively, we find that 
    \[\rho(\op{fr}_\ell)=\mtx{\ell x}{}{}{z}\text{ and }\rho(\tau_\ell)=\mtx{1}{y}{}{1}.\] The relation $\op{fr}_\ell \tau_\ell \op{fr}_\ell^{-1}=\tau_\ell^\ell$ then implies that $x=z$. 
\par Suppose that for some $n\geq 1$, 
\[\rho_n(\op{fr}_\ell)=\mtx{\ell x_n}{}{}{z_n}\text{ and }\rho_n(\tau_\ell)=\mtx{1}{y_n}{}{1}.\]
It suffices to lift $A_n$ to $A_{n+1}\in \widehat{\op{GL}}_2(\cO/\varpi^{n+1})$, so that 
\[\rho_{n+1}(\op{fr}_\ell)=\mtx{\ell x_{n+1}}{}{}{z_{n+1}}\text{ and }\rho_{n+1}(\tau_\ell)=\mtx{1}{y_{n+1}}{}{1},\]
where $x_{n+1}, y_{n+1}, z_{n+1}$ lift $x_n, y_n$ and $z_n$ respectively.
Let $B_{n+1}\in \widehat{\op{GL}}_2(\cO/\varpi^{n+1})$ be a lift of $A_n$, and set $r_{n+1}:=B_{n+1} \rho_{n+1}' B_{n+1}^{-1}$. Note that $\rho_n=r_{n+1}\mod{\varpi^n}$. In particular, $r_{n+1}(\op{fr}_\ell)$ lifts $\rho_n(\op{fr}_\ell)=\mtx{\ell x_n}{}{}{z_n}$ and $r_{n+1}(\tau_\ell)$ lifts $\rho_n(\tau_\ell)=\mtx{1}{y_n}{}{1}$. Choose lifts $\tilde{x}_n, \tilde{y}_n, \tilde{z}_n$ of $x_n, y_n, z_n$ to $\cO/\varpi^{n+1}$. Then, $\rho_n(\op{fr}_\ell)$ is $\mtx{\ell \tilde{x}_n}{}{}{\tilde{z}_n}\mod{\varpi^n}$, and this just means that there is $\mtx{a}{b}{c}{d}$ with coefficients determined mod-$\varpi$ such that \[r_{n+1}(\op{fr}_\ell)=\mtx{\ell \tilde{x}_n}{}{}{\tilde{z}_n}+\varpi^n\mtx{a}{b}{c}{d}.\] Similarly, there is $\mtx{e}{f}{g}{h}$ with coefficients determined mod-$\varpi$ such that 
    \[r_{n+1}(\tau_\ell)=\mtx{1}{\widetilde{y}_n}{0}{1}+\varpi^n\mtx{e}{f}{g}{h}.\] These coefficients $a,b,c,d,e,f,g,h\in \cO/\varpi$ are determined by $r_{n+1}$, i.e., they are determined by $\rho_{n+1}'$ and the choice of matrix $B_{n+1}$. Let \[B:=\left(\op{Id}+\varpi^n B'\right)\in \widehat{\op{GL}}_2(\cO/\varpi^{n+1}),\] and we set $r_{n+1}':=B r_{n+1} B^{-1}$. Write $B'=\mtx{a'}{b'}{c'}{d'}$, in due course, the values of $a',b',c'$ and $d'$ shall be prescribed.
    \par Let $M\in \op{GL}_2(\cO/\varpi^{n+1})$ and $B:=\op{Id}+\varpi^n B'$ in $\op{ker}\pi_n$. It follows from Lemma \ref{BMB^{-1}lemma} that  \[B M B^{-1}=M+\varpi^n [B',M].\]
 If $M\equiv \mtx{\eta\ell}{}{}{\eta}\mod{\varpi}$, then, 
    \begin{equation}\label{neweqn1}B MB^{-1}=M+\varpi^n\mtx{}{\eta(1-\ell)b'}{\eta(\ell-1)c'}{}.\end{equation}
    On the other hand, if $M\equiv \op{Id}\mod{\varpi}$, then, $[B',M]=0$, and \begin{equation}\label{neweqn2}B M B^{-1}=M.\end{equation} 
From the above computations, we find that 
    \[\begin{split} &r_{n+1}'(\op{fr}_\ell)=\mtx{\ell \tilde{x}_n}{}{}{\tilde{z}_n}+\varpi^n\mtx{a}{b+\eta(1-\ell)b'}{c+\eta(\ell-1)c'}{d};\\ 
    & r_{n+1}'(\tau_\ell)=\mtx{1}{\widetilde{y}_n}{0}{1}+\varpi^n\mtx{e}{f}{g}{h}.\end{split}\]
    The first of the above equations follows from \eqref{neweqn1} taking $M:=r_{n+1}(\op{fr}_\ell)$ and the second follows from \eqref{neweqn2} upon taking $M:=r_{n+1}(\tau_\ell)$.
    Since $\ell\not\equiv 1\mod{p}$, we may choose $B$ so that $r_{n+1}'(\op{fr}_\ell)$ is diagonal. In greater detail, $b':=\frac{b}{\eta(\ell-1)}$ and $c':=\frac{c}{\eta(1-\ell)}$, and set $a'=0$ and $d'=0$. Thus, we write \[r_{n+1}'(\op{fr}_\ell)=\mtx{\ell x_{n+1}}{}{}{z_{n+1}}.\]
    We find that 
    \[\begin{split} r_{n+1}'(\op{fr}_\ell\tau_\ell \op{fr}_\ell^{-1})= & \mtx{\ell x_{n+1}}{}{}{z_{n+1}}\mtx{1+\varpi^n e}{\widetilde{y}_n+\varpi^n f}{\varpi^n g}{1+\varpi^n h}\mtx{\ell x_{n+1}}{}{}{z_{n+1}}^{-1}, \\
   = & \mtx{1+\varpi^n e}{\ell x_{n+1}z_{n+1}^{-1}\left(\widetilde{y}_n+\varpi^n f\right)}{\varpi^n \ell^{-1}z_{n+1}x_{n+1}^{-1}g}{1+\varpi^n h},\\
   = & \mtx{1+\varpi^n e}{\ell x_{n+1}z_{n+1}^{-1}\widetilde{y}_n+\varpi^n \ell f}{\varpi^n \ell^{-1}g}{1+\varpi^n h},\\
   r_{n+1}'(\tau_\ell)^\ell=& \mtx{1+\varpi^n \ell e}{\ell \widetilde{y}_n+\varpi^n \ell f}{\varpi^n \ell g}{1+\varpi^n \ell h}.
    \end{split}\]

   We note that $x_{n+1}\equiv z_{n+1}\equiv \eta\mod{\varpi}$, and thus $x_{n+1}z_{n+1}^{-1}\equiv 1\mod{\varpi}$. The coefficients of the above matrices are in $\cO/\varpi^{n+1}$, and hence, the relations \[\varpi^n x_{n+1}z_{n+1}^{-1}=\varpi^n\text{ and }\varpi^n z_{n+1}x_{n+1}^{-1}=\varpi^n\] are used in the above equations. Since $\ell\not \equiv \pm 1\mod{p}$, we find that $e=g=h=0$. Hence, we find that \[\begin{split} &r_{n+1}'(\op{fr}_\ell)=\mtx{\ell x_{n+1}}{}{}{z_{n+1}};\\ 
    & r_{n+1}'(\tau_\ell)=\mtx{1}{y_{n+1}}{}{1}.\end{split}\]
    We set $A_{n+1}:=B B_{n+1}$, we note that $A_n\equiv B_{n+1}\mod{\varpi^n}$, and $B\equiv \op{Id}\mod{\varpi^n}$. With respect to this choice, $\rho_{n+1}:=r_{n+1}'$, and $A_n=A_{n+1} \mod{\varpi^n}$. This completes the inductive lifting argument. By previous remarks in the proof, this is enough to establish the result. 
\end{proof}

\begin{lemma}\label{lemma 4.5}
Let $\ell\nmid N_gp$ be a prime that divides $N_f$. Then, the following assertions hold.
    \begin{enumerate}
        \item\label{lemma 4.5 p1} If $\ell\in \Pi_g$, then, $\sigma_\ell(g)=1$ and $\sigma_\ell(f)=0$.
        \item\label{lemma 4.5 p2} If $\ell\in \Omega_g$, then, $\sigma_\ell(g)=0$ and $\sigma_\ell(f)=0$.
    \end{enumerate} 
\end{lemma}

\begin{proof}
    We begin by proving part \eqref{lemma 4.5 p1}. Since $\rho_g$ is unramified at $\ell$, we find that $V'_g=V_g$. Since $\bar{\rho}_g(\op{fr}_\ell)=\mtx{\ell}{}{}{1}$, it follows that $\widetilde{P}_\ell(g;X)=(1-X)(1-\ell X)$. Since $\ell\not\equiv 1\mod{p}$, it follows that $\ell^{-1}$ is a root of $\widetilde{P}_\ell(g;X)$ with multiplicity $1$. Therefore, $d_\ell(g)=1$. On the other hand, it follows from Lemma \ref{lemma 4.4} that $V_f'$ is the trivial $\op{G}_\ell$-module, we find that $\widetilde{P}_\ell(f;X)=1-X$, and $\ell^{-1}$ is not a root of $\widetilde{P}_\ell(f;X)$. Therefore, $d_\ell(f)=0$. Since $\ell^{p-1}\not\equiv 1\mod{p^2}$, it follows that that $s_\ell=1$. Since $\sigma_\ell(\cdot)=s_\ell d_\ell(\cdot)$, we find that $\sigma_\ell(g)=1$ and $\sigma_\ell(f)=0$.
    \par Next, we prove assertion \eqref{lemma 4.5 p2}. Since $\bar{\rho}_g(\op{fr}_\ell)=\mtx{-\ell}{}{}{-1}$, it follows that $\widetilde{P}_\ell(g;X)=(1+X)(1+\ell X)$. Thus, the roots of $\widetilde{P}_\ell(g;X)$ are $\{-1, -\ell^{-1}\}$. Since $\ell \not\equiv -1\mod{p}$, we find that $\ell^{-1}$ is not a root of $\widetilde{P}_\ell(g;X)$, and therefore, $d_\ell(g)=0$. Hence, it follows that $\sigma_\ell(g)=0$. As in the proof of part \eqref{lemma 4.5 p1}, it follows that $\widetilde{P}_\ell(f;X)=(1+X)$. Therefore, $\ell^{-1}$ is not a root of this polynomial, and it thus follows that $\sigma_\ell(f)=0$.
\end{proof}

Let $Q=\{q_1, \dots, q_n\}$ be a set of primes contained in $\Pi_g$ and $Q'=\{\ell_1, \dots, \ell_r\}$. For $(n, r)\in \Z_{\geq 0}^2$, we set $\frak{T}(n, r)$ be the collection of all sets $\Sigma=Q\cup Q'$, where $Q=\{q_1, \dots, q_n\}$ (resp. $Q'=\{\ell_1, \dots, \ell_r\}$) is a subset of $\Pi_g$ (resp. $\Omega_g$). The understanding is that when $n=0$ (resp. $r=0$), the set $Q$ (resp. $Q'$) is empty. When we write $\Sigma=\{q_1, \dots, q_n, \ell_1, \dots, \ell_r\}$, we shall implicitly mean that $\{q_1, \dots, q_n\}$ is contained in $\Pi_g$ and $\{\ell_1, \dots, \ell_r\}$ is contained in $\Omega_g$.

\begin{definition}
    For $\Sigma\in \frak{T}(n, r)$, set $N_\Sigma:=\prod_{i=1}^n q_i\times \prod_{j=1}^r \ell_j$. Let $\cF(\Sigma)$ be the set of newforms $f$ of weight $2$ such that 
    \begin{enumerate}
        \item $\bar{\rho}_f\simeq \bar{\rho}_g$, 
        \item $N_f=N_g N_\Sigma$.
    \end{enumerate}
\end{definition}

\begin{proposition}
    Let $g$ be a Hecke newform of optimal weight $k=2$ and optimal level. Then, for $(n, r)$ such that $n>0$ or $r>0$, and $\Sigma\in \frak{T}(n, r)$, the set $\cF(\Sigma)$ is nonempty.
\end{proposition}

\begin{proof}
    The result is a direct consequence of Theorem \ref{DT theorem}.
\end{proof}

\begin{lemma}\label{lemma 4.8}
    Let $f\in \cF(\Sigma)$, and let $\ell\neq p$ be a prime which divides $N_g$. Then, $\sigma_\ell(f)=\sigma_\ell(g)$.
\end{lemma}
\begin{proof}
    It suffices for us to show that $d_\ell(f)=d_\ell(g)$. Let $h\in\{g,f\}$, recall that $V_h':=(V_h)_{\op{I}_\ell}$. Note that as modules over the inertia group $\op{I}_\ell$, $V_h$ and $A_h[\varpi]$ are self dual. Since $\op{ord}_\ell(N_f)=\op{ord}_\ell(N_g)$, and $N_g$ is the prime to $p$ part of the Artin-conductor of $\bar{\rho}$, it follows that \begin{equation}\label{dim equality}\op{dim} A_h[\varpi]_{\op{I}_\ell}=\op{dim} (V_h)_{\op{I}_\ell},\end{equation}(cf. \cite[proof of Lemma 4.1.2]{emerton2006variation}). Recall that $P_\ell(h;X):=\op{det}\left(\op{Id}-\op{fr}_\ell X\mid V_h'\right)$ and that $\widetilde{P}_\ell(h;X)$ is the mod-$\varpi$ reduction of $P_\ell(h;X)$. Set $T_h':=(T_h)_{\op{I}_\ell}$. We identify $A_h[\varpi]$ with $T_h/\varpi T_h$, and thus $A_h[\varpi]_{\op{I}_\ell}$ with $T_h'/\varpi T_h'$. The equality \eqref{dim equality} implies that $T_h'$ is torsion free. We identify $T_h'\otimes_{\cO} K$ with $V_h'$, and since $T_h'$ is torsion free, we find that $T_h'$ is an $\cO$-lattice in $V_h'$, and that $P_\ell(h;X)=\op{det}\left(\op{Id}-\op{fr}_\ell X\mid T_h'\right)$. Therefore, the mod-$\varpi$ reduction of $P_\ell(h;X)$ is given by \[\begin{split}\widetilde{P}_\ell(h;X) =& \op{det}\left(\op{Id}-\op{fr}_\ell X\mid \left(T_h'/\varpi T_h'\right)\right),\\
    =& \op{det}\left(\op{Id}-\op{fr}_\ell X\mid A_h[\varpi]_{\op{I}_\ell}\right),\\ \end{split}\]
Since $A_g[\varpi]\simeq A_f[\varpi]$ as $\op{G}_\ell$-modules, we find that $A_g[\varpi]_{\op{I}_\ell}\simeq A_f[\varpi]_{\op{I}_\ell}$ as $\op{G}_\ell/\op{I}_\ell$-modules. It thus follows that $\widetilde{P}_\ell(f;X)=\widetilde{P}_\ell(g;X)$, from which we deduce that $d_\ell(f)=d_\ell(g)$. 
\end{proof}

\begin{theorem}\label{thm 4.9}
 Let $g$ be a Hecke newform of optimal level $N_g$, trivial nebentype and weight $2$. Assume that the following conditions are satisfied
\begin{enumerate} 
\item $\bar{\rho}_g$ is absolutely irreducible,
\item $g$ is $p$-ordinary and $p\nmid N_g$,
\item $g$ has optimal level,
\item $\mu_p(g)=0$.
\end{enumerate}
   Let $\Sigma\in \mathfrak{T}(n, r)$ and $f\in \cF(\Sigma)$. Then, the following assertions hold
   \begin{enumerate}
       \item $f$ has good ordinary reduction at $p$,
       \item $\mu_p(f)=0$,
       \item $\lambda_p(f)=\lambda_p(g)+n$.
   \end{enumerate}
\end{theorem}
\begin{proof}
    As noted earlier, that $f$ has ordinary reduction at $p$ follows from \cite[Lemma 3.3]{ray2023constructing}. From Theorem \ref{GV thm}, we find that $\mu_p(f)=0$ and 
    \[\lambda_p(f)=\lambda_p(g)+\sum_{i=1}^n \left(\sigma_{q_i}(g)-\sigma_{q_i}(f)\right)+\sum_{j=1}^r \left(\sigma_{\ell_j}(g)-\sigma_{\ell_j}(f)\right)+\sum_{\ell|N_g}\left(\sigma_{\ell}(g)-\sigma_{\ell}(f)\right). \]
    For $\ell|N_g$, it follows from Lemma \ref{lemma 4.8} that 
    \[\sum_{\ell|N_g}\left(\sigma_{\ell}(g)-\sigma_{\ell}(f)\right)=0.\]
    It follows from \eqref{lemma 4.5 p1} of Lemma \ref{lemma 4.5} that 
    \[\left(\sigma_{q_i}(g)-\sigma_{q_i}(f)\right)=1\] and it follows from \eqref{lemma 4.5 p2} of Lemma \ref{lemma 4.5} that 
    \[\left(\sigma_{\ell_j}(g)-\sigma_{\ell_j}(f)\right)=0.\]
    It therefore follows that 
    \[\lambda_p(f)=\lambda_p(g)+n.\]
\end{proof}

\section{Constructing Galois representations with prescribed $\lambda$-invariant}\label{s 5}
\par Throughout this section, $p\geq 5$. We introduce our assumptions. Let $g$ be a normalized newform of weight $2$ on $\Gamma_0(N_g)$. Throughout this section, we assume the following conditions.

\begin{enumerate} 
\item The residue field $\kappa=\cO/\varpi$ is $\F_p$, i.e., $f(\mathfrak{p}/p)=1$ where $\mathfrak{p}$ is the prime above $p$ prescribed by $\iota_p$.  
\item The Galois representation
$\bar{\rho}_g:\op{Gal}(\bar{\Q}/\Q)\rightarrow \op{GL}_2(\F_p)$ is surjective.
\item The modular form $g$ is $p$-ordinary and $p\nmid N_g$,
\item $g$ has optimal level,
\item $\mu_p(g)=0$.
\end{enumerate}

We show that the sets $\Pi_g$ and $\Omega_g$ both have positive density. Furthermore, we estimate these densities. We let $Y$ (resp. $Y'$) be the subset of $\op{GL}_2(\F_p)$ consisting of semisimple matrices conjugate to $\mtx{a}{}{}{1}$ (resp. $\mtx{a}{}{}{-1}$), where $a\neq \pm 1$. It is easy to see that 
\[\# Y =\# Y'= (p-3)\frac{\# \op{GL}_2(\F_p)}{\# \op{T}(\F_p)},\] where $\op{T}$ denotes the diagonal torus. Therefore, we find that 
\begin{equation}\label{card of Y and Y'}
    \frac{\# Y }{\# \op{GL}_2(\F_p)}=\frac{\# Y'}{\# \op{GL}_2(\F_p)}=\frac{(p-3)}{(p-1)^2}.
\end{equation}
Let $\bar{\rho}$ denote the residual representation $\bar{\rho}_g$, and let $\Q(\bar{\rho})$ be the Galois extension of $\Q$ which is the fixed field of the kernel of $\bar{\rho}$. We refer to $\Q(\bar{\rho})$ as the field \emph{cut out by $\bar{\rho}$}. We set $G$ to denote the Galois group $\op{Gal}(\Q(\bar{\rho})/\Q)$. The residual representation induces an isomorphism $\Phi: G\xrightarrow{\sim} \op{GL}_2(\F_p)$. Let $Z$ (resp. $Z'$) denote $\Phi^{-1}(Y)$ (resp. $\Phi^{-1}(Y')$). It follows from \eqref{card of Y and Y'} that 
\begin{equation}\label{card of Z and Z'}
    \frac{\# Z }{\# G}=\frac{\# Z'}{\# G}=\frac{(p-3)}{(p-1)^2}.
\end{equation}

\begin{lemma}\label{lemma 5.1}
    Let $\ell\nmid N_g p$ be a prime. The following assertions hold
    \begin{enumerate}
        \item $\op{fr}_\ell\in Z$ if and only if $\bar{\rho}(\op{fr}_\ell)=\mtx{\ell}{}{}{1}$ up to conjugation and $\ell\not\equiv \pm 1\mod{p}$;
         \item $\op{fr}_\ell\in Z'$ if and only if $\bar{\rho}(\op{fr}_\ell)=\mtx{-\ell}{}{}{-1}$ up to conjugation and $\ell\not\equiv \pm 1\mod{p}$.
    \end{enumerate}
\end{lemma}
\begin{proof}
    Suppose that $\op{fr}_\ell\in Z$ (resp. $\op{fr}_\ell\in Z'$). Then, since the weight of $g$ is $2$ and has trivial nebentype, we find that $\op{det}\bar{\rho}(\op{fr}_\ell)=\ell$. It follows that $\bar{\rho}(\op{fr}_\ell)=\mtx{\ell}{}{}{1}$ (resp. $\bar{\rho}(\op{fr}_\ell)=\mtx{-\ell}{}{}{-1}$) and $\ell\not\equiv \pm 1\mod{p}$.
\end{proof}

\begin{proposition}\label{density of Omega_g}
    Let $\ell\nmid N_g p$ be a prime. Then, $\op{fr}_\ell\in Z'$ if and only if $\ell\in \Omega_g$. As a consequence, it follows that $\Omega_g$ has positive density equal to $\frac{(p-3)}{(p-1)^2}$.
\end{proposition}
\begin{proof}
    We find that $\op{fr}_\ell\in Z'$ if and only if $\bar{\rho}(\op{fr}_\ell)=\mtx{-\ell}{}{}{-1}$ up to conjugation and $\ell\not\equiv \pm 1\mod{p}$. In other words, $\op{fr}_\ell\in Z'$ if and only if $\ell\in \Omega_g$. By the Chebotarev density theorem, $\Omega_g$ has positive density equal to $\frac{\# Z'}{\# G}=\frac{(p-3)}{(p-1)^2}$.
\end{proof}

Recall that $\Q_1$ is the $\Z/p\Z$-extension of $\Q$ contained in $\Q_{\op{cyc}}$. Let $\cL/\Q$ be the compositum $\Q(\bar{\rho})\cdot \Q_1$. 
\begin{lemma}
    The extensions $\Q(\bar{\rho})$ and $\Q_1$ are linearly disjoint. 
\end{lemma}
\begin{proof}
    Let $E:=\Q(\bar{\rho})\cap \Q_1$ and $N:=\op{Gal}(\Q(\bar{\rho})/E)$. Note that $N$ is a normal subgroup of $G\simeq \op{GL}_2(\F_p)$ and that $[G:N]$ divides $p$. For $p\geq 5$, the group $\op{PSL}_2(\F_p)$ is simple. It is easy to see that $\op{GL}_2(\F_p)$ does not contain an index $p$ normal subgroup. Therefore, $N=G$ and thus, $\Q(\bar{\rho})$ and $\Q_1$ are linearly disjoint.
\end{proof}

Let $\Gamma_1:=\op{Gal}(\Q_1/\Q)$ and we find that $\op{Gal}(\cL/\Q)\simeq G\times \Gamma_1$. Let $W$ be the product $Z\times (\Gamma_1\backslash \{0\})$. 

\begin{proposition}\label{density of Pi_g}
    Let $\ell\nmid N_gp$ be a prime. Then, $\op{fr}_\ell\in W$ if and only if $\ell\in \Pi_g$. As a result, $\Pi_g$ has density equal to $\frac{(p-3)}{p(p-1)}$.
\end{proposition}
\begin{proof}
    Note that the prime $\ell$ is nonsplit in $\Q_1$ if and only if $\ell^{p-1}\not \equiv 1\mod{p^2}$. The result is therefore a direct consequence of Lemma \ref{lemma 5.1}. By the Chebotarev density theorem, $\Pi_g$ has density equal to
    \[\frac{\#Z\times (\#\Gamma_1-1)}{\#G \times \#\Gamma_1}=\frac{(p-3)}{p(p-1)}.\]
\end{proof}

\begin{proof}[Proof of Theorem \ref{main thm 1}]
    The Theorem is a direct consequence of Proposition \ref{density of Omega_g}, Proposition \ref{density of Pi_g} and Theorem \ref{thm 4.9}.
\end{proof}

\begin{theorem}\label{last thm}
     Let $p\geq 5$ be a prime and let $g$ be a normalized newform of weight $2$ on $\Gamma_0(N_g)$ satisfying all of the conditions of Theorem \ref{main thm 1}. Furthermore, assume that $\lambda_p(g)\leq 1$. Then, for any set of primes $\{\ell_1, \dots, \ell_r\}\subset \Omega_g$, there is a Hecke newform of weight $2$ of level $N_f=N_g\ell_1\dots \ell_r$ such that 
     \begin{enumerate}
    \item $f$ has good ordinary reduction at $p$, 
    \item $\bar{\rho}_g\simeq \bar{\rho}_f$, 
    \item $\mu_p(f)=0$, 
    \item $\op{rank}^{\op{BK}}(f)=\lambda_p(g)$.
\end{enumerate}
\end{theorem}

\begin{proof}
    It follows from Theorem \ref{main thm 1} that there exists $f$ satisfying the first three assertions, and such that $\lambda_p(f)=\lambda_p(g)$. Proposition \ref{prop 3.2} implies that 
    \[\op{rank}^{\op{BK}}(f)=\lambda_p(f)=\lambda_p(g).\] This proves the last assertion.
\end{proof}

\begin{proof}[Proof of Theorem \ref{main thm 2}]
    It suffices to show that there is a non-CM Hecke newform $g$ of weight $2$ on $\Gamma_0(N_g)$ and a density $1$ set of primes $\Sigma$ such that for all $p\in \Sigma$,
    \begin{enumerate}
        \item $p\geq 5$, 
        \item The residue field $\kappa=\cO/\varpi$ is $\F_p$, i.e., $f(\mathfrak{p}/p)=1$. 
\item The Galois representation
$\bar{\rho}_g:\op{Gal}(\bar{\Q}/\Q)\rightarrow \op{GL}_2(\F_p)$ is surjective.
\item The modular form $g$ is $p$-ordinary and $p\nmid N_g$,
\item $g$ has optimal level, i.e., $N_g$ is the prime to $p$ part of the Artin conductor of the residual representation,
\item $\mu_p(g)=0$ and $\lambda_p(g)=0$.
    \end{enumerate}
It then follows from Theorem \ref{main thm 1} that for $p\in \Sigma$, there are infinitely many Hecke newforms of weight $2$ such that 
 \begin{enumerate}
    \item $f$ has good ordinary reduction at $p$, 
    \item $\bar{\rho}_f$ is surjective, 
    \item $\mu_p(f)=0$, 
    \item $\lambda_p(f)=n$.
\end{enumerate}
Let $E_{/\Q}$ be any non-CM elliptic curve with Mordell-Weil rank $0$, and let $g$ be the associated Hecke newform. Consider the following observations.
\begin{itemize}
    \item The field of Fourier coefficients of $g$ is $\Q$, since it is associated to an elliptic curve over $\Q$. In particular, it follows that the residue field $\kappa$ is isomorphic to $\F_p$ for all primes $p$.
    \item The set of primes at which $E$ has good ordinary reduction has density $1$, by a result of Serre \cite{serre1981quelques}.
    \item Serre's open image theorem \cite{serre1972proprietes} shows that for all but finitely many primes $p$, the residual representation $\bar{\rho}_g:\op{Gal}(\bar{\Q}/\Q)\rightarrow \op{GL}_2(\F_p)$ is surjective. 
    \item If at a given prime $p$, $N_g$ is not optimal, then there must be a $p$-congruence between $g$ and a newform $h$ of weight $2$ and strictly smaller level. This follows from Ribet's level lowering theorem. There are only finitely many newforms of weight $2$ with level strictly less than $N_g$. Also, there are only finitely many primes $p$ for which two newforms are $p$-congruent. Therefore, for all but finitely many primes $p$, the level $N_g$ is optimal. 
    \item Finally, it follows from a theorem of Greenberg \cite[Theorems 4.1, 5.1]{Gre99} that for a density $1$ set of primes $p$, $\mu_p(g)=\lambda_p(g)=0$. \end{itemize}
    Therefore, there is a set of primes $\Sigma$ of density $1$ such that the above conditions are satisfied. This completes the proof.
\end{proof}

\bibliographystyle{alpha}
\bibliography{references}
\end{document}